\documentclass[11pt,regno]{amsart}
\usepackage{amsmath}
\usepackage{amsfonts,amssymb}

\usepackage{euscript,graphicx}
\usepackage{epstopdf,rotating}
\newtheorem{theorem}{Theorem}
\newtheorem{lemma}{Lemma}
\newtheorem{corollary}{Corollary}
\newtheorem{proposition}{Proposition}
\theoremstyle{definition}
\newtheorem{example}{Example}
\newtheorem{remark}{Remark}
\newtheorem*{remark*}{Remark}

\newcommand{\eqdef}{\stackrel{\scriptscriptstyle\rm def}{=}}

\def\bR{\mathbb{R}}

\def\cB{\EuScript{B}}

\def\cM{\EuScript{M}}

\def\cR{{\mathcal R}}
\def\bT{\mathbb{T}}
\newcommand{\essup}{\operatornamewithlimits{\text{\rm{\,ess\,sup}}}}
\def\OO{\mathcal O}

\DeclareMathSymbol{\varnothing}{\mathord}{AMSb}{"3F}

\begin{document}
\title{{(Non)Invariance of dynamical quantities for orbit equivalent flows}}

\author{Katrin~Gelfert} 
\address{Instituto de Matem\'atica, UFRJ, 
Cidade Universit\'aria - Ilha do Fund\~ao, Rio de Janeiro 21945-909,  Brazil}
\email{gelfert@im.ufrj.br}
\author{Adilson~E.~Motter}
\address{Department of Physics and Astronomy \&
	Northwestern Institute on Complex Systems,
  Northwestern University, Evanston, Illinois 60208-3112, USA}
\email{motter@northwestern.edu}

\begin{abstract}
	We study how dynamical quantities such as Lyapunov exponents,
        metric entropy, topological pressure, recurrence rates, and
        di\-men\-sion-like characteristics change under a time reparameterization of a
        dynamical system.
        These quantities are shown to either remain invariant, transform according to a
        multiplicative factor or transform through a convoluted dependence that
        may take the form of an integral over the initial local values. We discuss the
        significance of these results for the apparent non-invariance  of chaos in general relativity
        and explore applications to the synchronization of equilibrium states and the elimination of 	expansions.
\keywords{Smooth time changes, Lyapunov exponents, thermodynamic formalism,
  dimensions, relativistic chaos}
\end{abstract}
\maketitle

\section{Introduction}
In the context of dynamical systems theory one may argue that the existence of
chaotic behavior together with its quantifiers should be invariant under changes of coordinates.
Depending on the quantifier of chaos, such changes could be either measurable,
continuous, or differentiable.
The topological entropy, for example, stands as a fundamental topological conjugacy
invariant.
The Hausdorff dimension, on the other hand,
 is invariant under bi-Lipschitz transformations of the phase space only.

In the case of continuous-time dynamical systems, as  generated by
flows of vector fields, the problem is more involved when one considers that
any change of coordinates allows in addition for a change of  time
parameterization. Some properties, such as ergodicity and the finiteness of the
entropy, are preserved by any measurable time transformation of the flow.
Other dynamical quantities or spectral invariants can exhibit surprising transformation
properties.
For example, suitable changes of time can strengthen
or simulate chaotic properties in the flow. In particular, for any ergodic flow there exists a transformation of time
that yields a mixing flow and for any ergodic flow with positive entropy there
exists a time change that yields a K-flow (see~\cite{KatTho:p} and references
therein).

In this paper, we offer a complete characterization of how quantifiers of chaotic
behavior transform under a smooth time reparameterization. The importance of
this work is two-fold. First, from a strictly mathematical perspective, it provides
a scheme of relationships for the classification and study of dynamical invariants
across orbit equivalent flows. For example, it is useful to consider the time
transformation of Lyapunov exponents and dimension quantities if one wants to
study the generality of conjectures that relate these quantities, such as the
Kaplan-Yorke relation~\cite{MotGel:09}.

Second, this work provides a mathematical foundation for the study of chaos in general
relativity and other physical theories that do not have an absolute time parameter.
The foundations of chaos in relativistic dynamics became a topic of intense research
after the observation, first reported in~\cite{FraMat:88}, that space-time transformations
commonly used in the study of (relativistic) cosmology could transform positive Lyapunov
exponents into zero Lyapunov exponents. This finding, which was confirmed in numerous
subsequent studies~\cite{BurBurEll:90,Ber:91,HobBerWelSim:91,Szy:97}  and was even the topic of a conference in the early nineties~\cite{HobBurCol:94}, led to the tacit assumption that Lyapunov exponents are unsuitable indicators of chaos in relativistic dynamics. Many alternatives have been proposed, including the use of methods based on symbolic dynamics~\cite{Rug:,CorLev:97}, Pailev\'e analysis~\cite{ConGraRam:94}, local curvature~\cite{SzyKra:96}, fractals~\cite{CorLev:97b,MotLet:01}, and system-specific analyses such as those based on Misner-Chitr\'e-like variables~\cite{BenMon:04}.
More recently, it has been shown that Lyapunov exponents can be restated as reliable
quantifiers of chaos if the transformation is such that it does not create singularities
in the invariant measure~\cite{Mot:03,MotSaa:09}. However, the conditions under which different approaches can be used to provide an invariant characterization of chaos under time reparameterizations remain debatable. In particular, it is currently unclear whether the fractal-based characterization of chaos by means of dimensions of measures, which has stood out as
the most promising alternative, is any more invariant than the Lyapunov exponents themselves.
This is an important motivation for establishing a mathematical basis for the study chaos under time reparameterizations, which is one of the goals of this paper.

Among the main results in the paper, we anticipate that for sufficiently regular
transformations: ($i$) the Lyapunov exponents
remain essentially invariant, except for a multiplicative factor that accounts for
the change of the time units; ($ii$) a time reparameterization can be used not only to
synchronize the measure of maximal entropy and the SRB measure, but also to
synchronize expansion rates and equilibrium states of smooth potentials;
($iii$) the generalized dimensions of measures $D_q$ are invariant for $q\ne 1$,
whereas the information dimension $D_1$ and entropies may change.
These and the accompanying results address a longstanding problem in relativistic dynamics, namely, they identify
the conditions under which different quantities can (or cannot) be used as invariant
indicators of chaos.

The paper is organized as follows. In Section~\ref{sec:2}, we introduce
definitions and recall basic results on time transformations, such as the
transformation of an invariant measure. In Section~\ref{sec:3}, we establish
a fundamental result on the time transformation of the solutions of the
variational equations.  In Section~\ref{sec:5}, we present the
main results about the transformation of dynamical quantities under time
reparameterizations. Results on the time transformation of dimension-like
characteristics  of measures are presented in Section~\ref{sec:6}. In the last
section we discuss implications of our results and future directions.

\section{Orbit equivalent flows and time reparameterizations}\label{sec:2}

Throughout the paper, $M$ is a manifold of smoothness $C^m$, $m\ge 2$, $g$
is a Riemannian metric of class $C^{m-1}$ on $M$, and $f$ is a $C^\ell$ vector
field on $M$, $1\le \ell<m$. We write $\rho$ for the distance induced by the
Riemannian structure on $M$.

We assume that the differential equation
\begin{equation}\label{eq1}
\frac{dx}{dt} = f(x)
\end{equation}
has a $C^\ell$ flow $\varphi\colon\bR\times M\to M$, and we write
$\varphi^t(\cdot)\eqdef\varphi(t,\cdot)$. The curve
$t\mapsto\varphi^t(x_0)$, $t\in\bR$, is the unique solution of~(\ref{eq1})
with the initial condition $x(0)=x_0$.
For the flow $\Phi=\{\varphi^t\}_{t\in\bR}$ we denote by $\cM(\Phi)$ the set
of all $\Phi$-invariant Borel probability measures on $M$.
We endow $\cM(\Phi)$ with the weak$*$ topology. Moreover, we denote by
$\cM_{\rm e}(\Phi)\subset \cM(\Phi)$ the subset of ergodic measures.

A $C^\ell$ flow $\Phi_r$ is a \emph{time reparameterization} of the flow $\Phi$ if for
every $x\in M$ the orbits
$\{\varphi^\tau_r(x)\}_{\tau\in\bR}$ and $\{\varphi^t(x)\}_{t\in\bR}$ coincide and if they have the same
orientation given by the change of $t$ and $\tau$ in the positive direction%
\footnote{Throughout the paper, the index $r$ is used to indicate time reparameterized quantities. The use of this notation will become clear shortly.}.
We study a $C^\ell$ flow $\Phi_r=\{\varphi^\tau_r\}_{\tau\in\bR}$ on $M$ that is
\emph{$C^k$-orbit equivalent} to $\Phi$, $k\le \ell$. This means that there exists a
$C^k$ diffeomorphism $h\colon M\to M$ such that the flow
$\{\widetilde\psi^\tau\}_{\tau\in\bR}$  given by $\widetilde
\psi^t\eqdef h^{-1}\circ \varphi^\tau_r\circ h$ is a time reparameterization of the
flow $\{\varphi^t\}_{t\in\bR}$.
Without loss of generality for the results of this paper we assume that $h$ is the identity, meaning that the orbits themselves are not deformed. See~\cite{Par:81,CorFomSin:82,KatHas:95,Tot:66} for treatments of the classical theory of time reparameterization of flows.

If $\Phi_r$ is a time transformation of $\Phi$, then we have
\begin{equation}\label{F1}
\varphi_r^\tau(x) = \varphi^{t(x,\tau)}(x)
\end{equation}
for every $x\in M$, where
$(x,\tau)\mapsto t(x,\tau)$ is a real-valued function which satisfies
$t(x,\tau_1+\tau_2)=t(x,\tau_1)+ t(\varphi^{\tau_1}(x),\tau_2)$,
$t(x,-\tau)=-t(\varphi^{-\tau}(x),\tau)$, and $t(x,\tau)\ge 0$
if $\tau\ge 0$.
For every point $x$ that is not a fixed point,  $(x,\tau)\mapsto t(x,\tau)$ is a $C^\ell$ function satisfying $t(x,\tau)>0$ whenever $\tau>0$.
One can view the flow $\Phi_r$ as being obtained from $\Phi$ by means of a
change of time parameterization determined by
\begin{equation}\label{eq1t}
	d\tau = r(x)\,dt,
\end{equation}
where
\[
r(x)\eqdef
\left(\frac{\partial}{\partial\tau}t(x,\tau)|_{\tau=0}\right)^{-1}
\]
at every non-fixed point $x$. In particular,
at every non-fixed $x$ the function $r$ is $C^{\ell-1}$ and positive, and
$\tau\mapsto \varphi^\tau_r(x)$ solves the differential equation
\begin{equation}\label{eq}
\frac{dx}{d\tau} = \frac{f(x)}{r(x)},
\end{equation}
with initial condition $x(0)=x$.
Of particular interest are orbit equivalent flows that are  \emph{a priori} obtained by means of a
time change determined by~(\ref{eq1t}). Clearly, to accomplish orbit equivalence, such a  function $r$  is non-negative everywhere, and positive and sufficiently regular at any non-fixed point.

As indicated above, the description of the behavior of the flows close to fixed points requires
particular care. To describe this in a general setting, assume that $r$ is a non-negative measurable function and denote by
\begin{equation}\label{defm}
	\widehat M
	\eqdef \left\{ x\in M\colon f(x)\ne 0\right\}
\end{equation}
the \emph{regular set} of $\Phi$ (and hence of $\Phi_r$).
In the particular case in which $r$ is  positive and of class $C^{\ell-1}$
everywhere on $M$, we set $\widehat M\eqdef M$ instead of~(\ref{defm}).
Given $x\in\widehat M$ and $t\in\bR$, we define
\begin{equation}\label{defah}
 \tau(x,t)\eqdef \int_0^t  r(\varphi^s(x))\,ds\,.
\end{equation}
If $\mu\in\cM(\Phi)$ and $r$ is $\mu$-integrable and positive on some set of positive measure, then
for almost every $x\in\widehat M$ the function $t\mapsto \tau(x,t)$ is strictly increasing and satisfies $ \tau(x,0)=0$ and
\begin{equation}\label{limbeta}
	\lim_{t\to\pm\infty}   \tau(x,t)= \pm\infty.
\end{equation}
Then $\Phi_r = \{\varphi^\tau_r\}_{\tau\in \bR}$ is a measurable flow on $(\widehat M,\widehat\cB)$, where $\widehat\cB\eqdef \cB\cap\widehat M$ is the restriction of the
$\sigma$-field $\cB$ to $\widehat M$.
Condition~(\ref{limbeta}) provides an important criterion for the selection of eligible time reparameterizations for the study of chaos in relativistic systems exhibiting cosmological singularities~\cite{MotLet:02} or event horizons~\cite{MotSaa:09}.

The time-transformed flow $\Phi_r$  in general evolves at a different speed and
possesses different invariant probability distributions.
Given some $\mu\in\cM(\Phi)$, assume that $r$ is a non-negative function that is positive on some positive measure set and satisfies $\int_0^1r(\varphi^s(x))ds<\infty$ at $\mu$-almost every $x$.
This measure given by $d\widehat\mu_r \eqdef r \, d\mu$
on $\widehat M$ is $\Phi_r$-invariant and $\sigma$-finite.
The measure is finite, and hence normalizable, if and only if $r$ is  $\mu$-integrable,
in which case we introduce the $\Phi_r$-invariant probability measure $\mu_ r$ by
\begin{equation}\label{sua}
d\mu_r \eqdef \frac{ r}{\int_M r \,d\mu} d\mu .
\end{equation}
Note that for every $\Phi$-invariant measurable $A$ we have $\mu_r(A)=0$ if
$\mu(A)=0$. 
The measure $\mu_ r$ is ergodic if $\mu$ is.
If $r$ is $\mu$-integrable, from~(\ref{sua}) we conclude that for every continuous function $\xi\colon
M\to\bR$ we have
\begin{equation}\label{defahh}
  \int_M \frac{\xi}{ r} \,d\mu_ r
  = \frac{\int_M\xi \,d\mu}{\int_M r \,d\mu}\,.
\end{equation}
For a proof of the above stated properties in the context of abstract
metrically isomorphic flows, see~\cite{CorFomSin:82,Tot:66} and
references therein.
We continue to denote by $\cM(\Phi\equiv\Phi_1)$ (by $\cM(\Phi_r)$) the set of
$\Phi$-invariant ($\Phi_r$-invariant) probability measures.

\begin{remark}{\rm
If the function $r$ is such that $1/r\in L^1(M,\mu)$, then the original flow
$\Phi$ is obtained by the time reparameterization of the flow $\Phi_r$  determined by the function $1/r$. In
this case the map $\mu\mapsto\mu_ r$ defined by~(\ref{sua}) is a bijection from
$\cM_{\rm e}(\Phi)$ to $\cM_{\rm e}(\Phi_r)$. The
decomposition of a measure into its ergodic components is preserved under this
transformation, although the particular weight associated to each component is not necessarily the same.
}\end{remark}

\begin{remark}{\rm
Given two orbit equivalent flows $\Phi$ and  $\Phi_r$ without fixed points, for every $\mu\in\cM_{\rm e}(\Phi)$ the sets of generic points with respect to the measures $\mu$ and $\mu_r$
coincide~\cite{Ohn:80}.
This means that we have $G(\mu,\Phi)=G(\mu_ r,\Phi_r)$, where
\[
  G(\mu,\Phi)\eqdef
  \left\{x\colon\lim_{t\to\infty}\frac{1}{t}\int_0^t\xi(\varphi^s(x))\,ds
    = \int_M \xi \,d\mu\,\,{\rm for}\,\,{\rm every}\,\,\xi\in C^0(M,\bR)\right\}
\]
and $G(\mu_ r,\Phi_r)$ is defined similarly.
}\end{remark}

\section{Variational equations}\label{sec:3}

Given a point $x\in M$ and the trajectory $\{\varphi^t(x)\colon
t\in\bR\}$ that passes through $x$, we introduce the \emph{variational
  differential equation}
\begin{equation}\label{ve}
  \frac{Dy}{dt}=\nabla f(\varphi^t(x))\,y,
\end{equation}
where $\nabla f$ is the covariant derivative of the vector field $f$. The
absolute derivative is taken along the curve $t\mapsto \varphi^t(x)$.
For $x\in M$ and $v\in T_xM$, the linearization $y(t)\eqdef y(t,x,v)$
of the flow map $\varphi^t$ given by
\begin{equation}\label{ve_vec}
y(t)\eqdef D_x\varphi^t(v)
\end{equation}
is the solution of equation~(\ref{ve}) with $y(0)=v$.
Analogously, given $x\in\widehat M$, the linearization $z(\tau)\eqdef z(\tau,x,v)$ of the flow map $\varphi_r^\tau$ given by
\[
z(\tau)\eqdef D_x\varphi_r^\tau(v)
\]
is the solution of the variational differential equation
\begin{equation}\label{kak}
  \frac{Dz}{d\tau}
  = \frac{1}{ r(\varphi_r^\tau(x))}\nabla f(\varphi_r^\tau(x)) \,z
  +  \left( D\frac{1}{r}(\varphi_r^\tau(x)) \cdot z \right)
     f(\varphi_r^\tau(x))
\end{equation}
with the initial condition $z(0)=v$.

\begin{proposition}\label{oi}
  For $x\in \widehat M$ and $v\in T_xM$, the maps $t\mapsto y(t)$ and $\tau\mapsto  z(\tau)$ given above satisfy
  \begin{equation}\label{anss}
    z(\tau(x,t)) = y(t) + \varkappa(t) \left(\frac{f}{r}\right)(\varphi^t(x))
  \end{equation}
  for every $t$, where the function $\varkappa\colon \bR\to\bR$ is given by
  \begin{equation}\label{nqlfs}
  \varkappa(t) =
  - \int_0^t D r(\varphi^s(x)) \cdot y(s) \,ds.
  \end{equation}
\end{proposition}

\begin{proof}
  By transforming the time in~(\ref{ve}) for $\gamma(\tau)\eqdef y(t(x,\tau))$
  and using relation~(\ref{F1}) and $dt(x,\tau)/d\tau=r(\varphi^{t(x,\tau)}(x))^{-1}$,
  we obtain
  \begin{eqnarray*}
      \frac{D \gamma}{d\tau}
      =
      \frac{dt}{d\tau}\frac{D\gamma}{dt}
        = \nabla\left( \frac{f}{r}\right)(\varphi_r^\tau(x))
         \,\gamma(\tau) -
      (Dr^{-1}(\varphi_r^\tau(x))
      \cdot \gamma(\tau))f(\varphi_r^\tau(x)).
       \end{eqnarray*}
  Making the ansatz
  $  z(\tau) =
  y(t(x,\tau)) + \varkappa(\tau)\left(f/r\right)(\varphi^{t(x,\tau)})
  $  for some function $\varkappa\colon \bR\to\bR$, a short
  calculation gives
  \begin{equation}\label{heu}
      \frac{D \gamma}{d\tau}
      = \nabla\left(\frac{f}{r}\right)(\varphi_r^\tau(x)) \,\gamma(\tau)
         -\dot\varkappa(\tau)\left(\frac{f}{r}\right)(\varphi_r^\tau(x)).
  \end{equation}
  From the above
  and~(\ref{heu}) we obtain that $\varkappa(\tau)$  satisfying
  $\varkappa(0)=0$ is given
  by
  \[
  \varkappa(\tau) = - \int_0^{\tau} D\log r(\varphi_r^s(x)) \cdot y(t(x,s)) \,ds
  \]
  for every $\tau$.
  After a change of variable, we obtain~(\ref{anss}).
\end{proof}

\begin{example}[Hyperbolic splitting]{\rm
Let us consider the example of a vector field $f$ with a $C^1$ Anosov flow $\Phi$
on a compact smooth manifold $M$. This means that there exists
$\lambda>0$ and a Riemannian metric such that for every $x\in M$ there
is a splitting $T_xM=E^u\oplus E^c\oplus E^s$ with $f(x)\in
E^c_x\setminus\{0\}$, $\dim E^c_x=1$,
$D_x\varphi^t E^{u/s}_x=E^{u/s}_{\varphi^t(x)}$ for all $t$, and $|| D_x\varphi^{\mp t}(v^{u/s})|| \le e^{-t\lambda}|| v^{u/s}||$  for every
$v^{u/s}\in E^{u/s}_x\setminus\{0\}$ and $t>0$.
By the Anosov-Sinai theorem~\cite{AnoSin:67}, a change of the time parameterization by means of a positive $C^1$ function $r$ produces again an Anosov flow.
Moreover, the splitting $T_xM=E^u_{x,r}\oplus E^c_{x,r}\oplus E^s_{x,r}$ that is invariant and hyperbolic with respect to the flow $\Phi_r$ can be described explicitly as follows (see~\cite{AnoSin:67,Par:86}). Given $x\in M$, we have
\[
E^{u/s}_{x,r}= \left\{ v + \kappa^{u/s}(v)f(x)\colon v\in E^{u/s}_x\right\} ,
\]
where
\[
\kappa^{u/s}(v) = -\frac{1}{r(x)}\int_0^{\mp\infty}  Dr(\varphi^t(x))\cdot D_x\varphi^t(v)\,dt.
\]
This integral converges uniformly since $v^{u/s}$ is in the unstable/stable subbundle $E^{u/s}_x$ and thus $\kappa^{u/s}$ is continuous. Note that this is in agreement with the proposition above, where~(\ref{anss}) describes how an arbitrary vector $v\in T_xM$ is stretched under $D\varphi_r^t$. Recall that $D_x\varphi_r^t(v)$ converges towards the direction of $E^{u/s}_{\varphi^t_r(x),r}$  as $t\to\pm\infty$  whenever $v\notin E^{s/u}_{x,r}\cup E^c_{x,r}$.
}\end{example}

\section{Transformation of dynamical quantities}\label{sec:5}

\subsection{Lyapunov exponents}\label{lyap}

Given $x\in M$ and $v\in T_xM$, the \emph{forward Lyapunov exponent} of $v$ at
$x$ with respect to the flow $\Phi$ is defined by
\begin{equation}\label{deflya}
  \lambda^+(x,v) \eqdef
  \limsup_{t\to\infty}\frac{1}{t}\log\,|| D_x\varphi^t(v)||,
\end{equation}
with the convention that $\log 0=-\infty$.
Here the  norm $||\cdot||$ is the one derived from the
scalar-product structure on the tangent bundle induced by the Riemannian
metric.
Analogously, the \emph{backward Lyapunov exponent} $\lambda^-(x,v)$ is defined
by replacing the map $\varphi^t$ in~(\ref{deflya}) with $\varphi^{-t}$.
For every point $x\in M$ that is \emph{Lyapunov regular} with respect
to $\Phi$, there exists a positive integer $k(x)\le \dim
M$ and a unique splitting
\begin{equation}\label{spki}
T_xM = \bigoplus_{i=1}^{k(x)}E^i_x
\end{equation}
of the tangent space, where each linear subspace $E^i_x$ depends measurably on
$x$ and satisfies $D_x\varphi^t(E^i_x) = E^i_{\varphi^t(x)}$ for every
$t\in\bR$. Moreover, there
are numbers $\lambda_1(x)< \cdots < \lambda_{k(x)}(x)$ such that for each
$v\in E^i_x$ we have
\[
\lim_{t\to\pm\infty}\frac{1}{t}\log\,|| D_x\varphi^t(v)|| =
\lambda^+(x,v) =-\lambda^-(x,v) = \lambda_i(x).
\]
In addition, for every $v\in E^i_x$ and $w\in E^j_x$ with $i\ne j$, we have
\begin{equation}\label{lyez}
  \lim_{t\to\pm\infty}\frac{1}{t}\log\,
  |\sin\angle(D_x\varphi^t(v),D_x\varphi^t(w))| =0,
\end{equation}
i.e., the angles between the two subspaces $E^i_x$ and $E^j_x$ can go to zero
at most sub-exponentially along the orbit of $x$ (see~\cite{BarPes:05} for
details on Lyapunov regularity).

 By the multiplicative ergodic
theorem, the set of Lyapunov regular points has full measure with respect to any $\mu\in \cM(\Phi)$.
Furthermore, the functions $x\mapsto \lambda_i(x)$, $\dim E^i_x$, and $k(x)$ are $\mu$-measurable
and $\Phi$-invariant $\mu$-almost everywhere, and hence are constant $\mu$-almost everywhere if $\mu$ is ergodic. Here invariant means that we have $\lambda_i(x)=\lambda_i(\varphi^s(x))$ for every $s$. The numbers $\lambda_i(x)$ are called the \emph{Lyapunov exponents} of $\Phi$ at $x$ and $\dim E^i_x$ is their multiplicity. For the remainder of the paper we use the same notation for the Lyapunov exponents to account for multiplicities:
\[
\lambda_1(x)\le\lambda_2(x)\le\ldots\le\lambda_{\dim M}(x).
\]
Given $\mu\in\cM(\Phi)$, we also consider
\begin{equation}\label{deflyame}
  \lambda_i(\mu) \eqdef \int_M\lambda_i(x)\, d\mu(x).
\end{equation}

Let us denote by $\cR\subset\widehat M$ the set of Lyapunov regular points such that  there exists a positive integer number $j(x) \le \dim M$ for which $\dim E^{j(x)}_x=1$ and $\lambda_{j(x)}(x)=\lambda(x,f(x))$ is the exponent in the direction of $f(x)$.
We now consider the case where
at every Lyapunov regular point $x\in M$ all Lyapunov
exponents except the one related to the direction of the vector field are non-zero.
More precisely, we assume that for every $x\in \cR$ there exists a number $j(x)$
such that $\dim E^{j(x)}_x=1$,  the exponent $\lambda_{j(x)}(x)$ is associated to the
direction of the flow,  and
\begin{equation}\label{thun}
\lambda_1(x)\le\cdots\le\lambda_{j(x)-1}(x) < \lambda_{j(x)}(x)=0
<\lambda_{j(x)+1}(x)\le\cdots\le\lambda_{\dim M}(x).
\end{equation}
This also includes the case that all exponents except the one related to the flow direction are positive (or negative).
A $\Phi$-invariant measure $\mu$ is called \emph{hyperbolic} if~(\ref{thun})
holds for $\mu$-almost every point $x$, in which case $\cR$ is necessarily of full measure.

Given $x\in \widehat M$, we define
\begin{equation}\label{nao}
   r_\pm(x)  \eqdef \lim_{t\to\pm\infty}\frac{1}{t} \int_0^t r(\varphi^s(x))\,ds
\end{equation}
whenever the limit exists.

\begin{proposition}\label{prop1}
  For every point $x\in \cR\subset\widehat M$ that is Lyapunov regular with respect to
  $\Phi$ and $\Phi_r$ and satisfies $0<r_\pm(x)<\infty$ and every $v\in T_xM\setminus \{0\}$   we have
  \begin{equation}\label{delma}
   \lim_{t\to\pm\infty}\frac{1}{t}
    \log\,|| D_x\varphi_r^{ \tau(x,t)}(v)||
    =
    \lim_{t\to\pm\infty}\frac{1}{t}
    \log\,|| D_x\varphi^t(v)||
     = \pm\,\lambda^\pm(x,v).
  \end{equation}
\end{proposition}

\begin{proof}
  Take $x\in \cR$ to be a point that is Lyapunov regular with respect to
  $\Phi$ and $\Phi_r$, and let $v\in T_xM\setminus \{0\}$.
  Denote $y(t)= D_x\varphi^t(v)$ and let  $\zeta(t)\eqdef
  D_x\varphi_r^{ \tau(x,t)}(v)$. From the Lyapunov regularity, it follows that
  the limits
  \begin{equation}\label{ham}
    \lambda^+(x,v)=
    -\lambda^-(x,v)=
    \lim_{t\to\pm\infty}\frac{1}{t}\log\,|| y(t)||
  \end{equation}
  exist.
  We first consider the case in which $v$ is in one of the remaining subspaces $E^i_x$, $i\ne j(x)$, which implies $\angle(v,f(x))\ne 0$. From relation~(\ref{anss}), it
  follows that
  \begin{equation}\label{alh}
  ||\zeta(t)|| =
  \frac{|\sin\angle(y(t),f(\varphi^t(x)))|}
       {|\sin\angle(\zeta(t),f(\varphi^t(x)))|}
       || y(t)||
  \end{equation}
  for every $t\in\bR$. Using~(\ref{lyez}), we get
  \begin{equation}\label{chwa}
  \lim_{t\to\pm\infty}\frac{1}{t}
  \log\,|\sin\angle(y(t),f(\varphi^t(x)))|
  =0.
  \end{equation}
  Note that for every $t\in\bR$ and every $x\in \widehat M$ we have
  $\varphi^t(x)=\varphi_r^{ \tau(x,t)}(x)$.
  Since $D_x\varphi_r^s((f/r)(x)) = (f/r)(\varphi_r^s(x))$ for every $s\in\bR$, we obtain
  \[
  |\sin\angle(D_x\varphi_r^{ \tau(x,t)}(v),
                   D_x\varphi_r^{ \tau(x,t)}(f(x)))|
  =  |\sin\angle(\zeta(t),f(\varphi^t(x)))|.
  \]
  From property~(\ref{lyez}) applied to the flow $\Phi_r$ we
  obtain
  \[
  \lim_{t\to\pm\infty}
  \frac{1}{ \tau(x,t)}\log\,
  |\sin\angle(\zeta(t),f(\varphi^t(x)))| = 0
  \]
  (note that $ \tau(x,-t)= -\int_0^t r(\varphi^{-s}(x))ds$).
  Using~(\ref{nao}) and $0<r_\pm(x)<\infty$, this leads to
  \begin{equation}\label{sko}
  \lim_{t\to\pm\infty}\frac{1}{t}\log\,
  |\sin\angle(\zeta(t),f(\varphi^t(x)))| = 0.
  \end{equation}
  Using (\ref{chwa}) and (\ref{sko}) in (\ref{alh}), we obtain
  \[
  \lim_{t\to\pm\infty}\frac{1}{t}\log\,||\zeta(t)||
  = \lim_{t\to\pm\infty}\frac{1}{t}\log\,|| y(t)||.
  \]
  In the other case, $v = f(x)$,  we have
  $  || D_x\varphi_r^{\tau(x,t)}(v)|| =
  r(x)/r(\varphi^t(x))
   || f(\varphi^t(x))||.
  $
  From $0<r_\pm(x)<\infty$, it follows that $\lim_{t\to\pm\infty}\frac{1}{t}\log
  r(\varphi^t(x))=0$.
  Hence, we obtain $\lambda^\pm(x,v) =\pm\lim_{t\to\pm\infty}\frac{1}{t}
  \log\,|| D_x\varphi_r^{\tau(x,t)}(v)|| $ for every vector $v\in
  T_xM\setminus\{0\}$ parallel to $f(x)$.
  The observation that an arbitrary $v$ can be split into components in subspaces $E^i_x$ completes the proof.
\end{proof}

\begin{remark}\label{rem:stra}{\rm
If condition $0<r_\pm(x)<\infty$ in Proposition~\ref{prop1} is relaxed, the statement remains valid  for any $v\in E^i_x$ and any $i\ne j(x)$, while for $v=f(x)$ the right-hand side of (\ref{delma})  has to be replaced by $ \pm\,\lambda^\pm(x,f(x))-R_\pm(x)$, where
    \[
    R_\pm(x) \eqdef \lim_{t\to\pm\infty}\frac{1}{t}\log r(\varphi^t(x)).
   \]
   Note that $R_\pm(x)\le0 $ if $r_\pm(x)<\infty$ and  $R_\pm(x)\ge0$  if $r_\pm(x)>0$.
   Moreover, if $\log r\in L^1(\mu)$ then $R_\pm(x)=0$ for $\mu$-almost every $x$.
}\end{remark}

Given a point $x$ that is Lyapunov regular with respect to the
flow $\Phi_r$, we use the notation
$\lambda^r_i(x)$ to distinguish the values of the Lyapunov exponents from those corresponding
to $\Phi$.

\begin{theorem}\label{theo.lyap}
For every point $x\in\cR\subset \widehat M$ that is Lyapunov regular with respect to
  $\Phi$ and $\Phi_r$ and satisfies $0<r_+(x)<\infty$, we have 	
  \begin{equation}\label{hi}
  \lambda^ r_i(x)\, r_+(x)= \lambda_i(x), \quad i=1,\ldots,\dim M,
  \end{equation}
  and hence the number of positive, null, and negative Lyapunov exponents at $x$ remains the same.
 \end{theorem}

\begin{proof}
Let $v\in T_xM$ and denote $\zeta(t) \eqdef
  D_x\varphi_r^{ \tau(x,t)}(v)$.
  Note that $ \tau(x,t)\to\infty$ as $t\to\infty$ and hence
  \begin{eqnarray*}
  \lim_{t\to\infty}\frac{1}{t}\log\,|| \zeta(t)||
  = r_+(x)
  \lim_{\tau\to\infty}\frac{1}{\tau}\log\,|| D_x\varphi_r^\tau(v)||.
  \end{eqnarray*}
  Lyapunov regularity and Proposition~\ref{prop1} imply that the Lyapunov exponents at $x$ with respect to $\Phi_r$  satisfy $\lambda_i^r(x) \, r_+(x)= \lambda_i(x)$.
   Note that the joint Lyapunov regularity automatically implies that $r_+(x)=r_-(x)$. This proves the theorem.
 \end{proof}

\begin{remark}\label{rem:strabn}{\rm
Let us consider the case in which the condition $0<r_+(x)<\infty$ in Theorem~\ref{theo.lyap} is relaxed.
It follows from the proof that for every Lyapunov regular point, $r_+(x)=0$ implies that we must have $\lambda_i(x) = 0$, $i=1$,$\ldots$, $\dim M$ (see also Example~\ref{ex:Lyap}). Analogously, if $r_+(x)=\infty$ then we have $\lambda_i^r(x) = 0$, $i=1$,$\ldots$, $\dim M$ (compare with Example~\ref{ex:suspen} below).
}\end{remark}

The following is a consequence of Theorem~\ref{theo.lyap}.

\begin{corollary}\label{cor.hyp}
  Let $\mu\in\cM(\Phi)$ be hyperbolic and  let $r\in L^1(M,\mu)$.
 For the measure $\mu_r$ given by (\ref{sua}) we have
  \begin{equation}\label{hott}
  \lambda_i^ r(\mu_ r)  \, {\int_M r \,d\mu} = \lambda_i(\mu), \quad
  i=1,\ldots,\dim M.
  \end{equation}
  Therefore, $\mu_r\in\cM(\Phi_r)$ is hyperbolic if $r$ is positive $\mu$-almost everywhere.
\end{corollary}

\begin{proof}
  In order to prove (\ref{hott}), what remains to
  be shown is that for every $i=1,\ldots,\dim M$ we have
  \begin{equation}\label{hutt}
  \int_M r \,d\mu\int_M\lambda_i^ r(x)\,d\mu_ r(x) = \lambda_i(\mu).
  \end{equation}
  By the Birkhoff  theorem, the  limit (\ref{nao}) exists $\mu$-almost everywhere and is finite, and $r_+\in L^1(M,\mu)$.
  By the multiplicative ergodic theorem, the set of Lyapunov regular points
  with respect to any of the flows $\Phi$ and $\Phi_r$ has
  full measure.
  By (\ref{sua}) we have
  \[
  \int_M r \,d\mu\int_M \lambda_i^r(x) \, d\mu_ r(x) =
  \int_M \lambda_i^r(x) \, r(x)\,d\mu(x).
  \]
  Recall that the function $x\mapsto r_+(x)$ is $\mu$-measurable and invariant according to the Birkhoff theorem,
  and
  that $x\mapsto \lambda_i(x)$ is $\mu$-measurable and invariant. Here
  invariance implies that we have $r_+(x)=r_+(\varphi^s(x))$ and
  $\lambda_i(x)=\lambda_i(\varphi^s(x))$ for every $s$, and thus
  \[
  \int_M r \,d\mu\int_M \lambda_i^ r(x)\,  d\mu_ r(x)
  = \int_M \lambda_i^r(x) \,  r(\varphi^s(x))\,d\mu(x).
  \]
  Further, by the first claim of the theorem, we can conclude that
  \begin{eqnarray*}
   &&\Big| \lambda_i(\mu) -
    \int_M \lambda_i^r(x)\, r(\varphi^s(x))\,d\mu(x)\Big|
    \\
    &&= \left| \int_M\lambda_i(x)\,d\mu(x) -
    \int_M\lambda_i^r(x)\, r(\varphi^s(x))\,d\mu(x)\right|
    \\
    &&= \left| \int_M\lambda_i^r(x)\, r_+(x)\,d\mu(x) -
    \int_M\lambda_i^r(x)\, r(\varphi^s(x))\,d\mu(x)\right|
    \\
    &&\le \int_M \lambda_i^r(x)   \left|  r_+(x) -
         \frac{1}{t}\int_0^t r(\varphi^s(x))\, ds\right| \,d\mu(x)
                   \to 0 \,\,{\rm as }\,\,t\to\infty.
  \end{eqnarray*}
   Thus it follows   that (\ref{hutt}) holds true, and this proves the theorem.
\end{proof}

We illustrate the above results with some explicit constructions.

\begin{example}[Lyapunov exponents for singular transformations]\label{ex:Lyap}{\rm
We consider the flow
        $\Phi=\{\varphi^t\}_{t\ge0}$ generated by the vector field
    $f(x)=-x| x|$ for every $x\in\bR$ and the time transformed flow
    $\Phi_r$ with the
    associated function $r(x)=| x|$, which has a
    stable fixed point $x=0$. For every $x\in \bR$ we have
    $\lambda(x)=0$ and ${\lambda^r}(x)=-1$. Note that $r_+(x)=0$ and
    $R_+(x)=0$ (this is not in conflict with Remark~\ref{rem:stra} since in (\ref{delma}) the $1/t$-factor  is not time-transformed).  This demonstrates that a singular time reparameterization with $r_+(x)=0$ may create non-zero Lyapunov exponents in the time-transformed flow.
}\end{example}

\begin{example}[Elimination of non-zero exponents]\label{ex:suspen}{\rm
	Let us consider the hyperbolic toral automorphism $h\colon\bT^2\to\bT^2$ of the two-torus given by $h(x,y)=(2x+y,x+y)\,{\rm mod}\, 1$. This flow is ergodic with respect to the Lebesgue measure $m$.
	Consider the mapping torus
	\[
	M\eqdef \bT^2\times[0,1] /\sim,
	\]
	where $\sim$ is the identification of $(x,y,1)$ with $(h(x,y),0)$, and consider the standard suspension flow of $h$ on $M$ that is defined by $\eta^t(x,y,s)\eqdef (x,y,s+t)$ for $0\le t+s< 1$. The flow $\{\eta^t\}_{t\in\bR}$  preserves the measure $\mu$ defined by
\[
\int_M\xi\,d\mu\eqdef\int_{\bT^2}\int_0^1\xi(x,y,t)\,dt\,dm(x,y)
\]
for  any continuous function $\xi\colon M\to\bR$, and $\mu$ is ergodic since $m$ is.
Note that the Lyapunov exponents (with respect to the flow $\{\eta^t\}_{t\in\bR}$) are given by
\[
\lambda_{1}(p)=(3-\sqrt{5})/2,\quad
\lambda_2(x)=0,\,\,{\rm and }\,\,\lambda_{3}(p)=(3+\sqrt{5})/2
\]
for every point $p\in M$.
Denote by  $(x_0,y_0)$  the fixed point of $h$ and let  $p_0= (x_0,y_0,0)$.
    Following the construction in~\cite{SunYouZho:06}, given some $\delta>0$ one can construct two $C^\infty$ functions $\omega_1$, $\omega_2\colon M\to[0,1]$ that satisfy
    \begin{enumerate}
    \item $\omega_{1}(p)=\omega_{2}(p)=0$ if and only if $p=p_0$,\\[-0.4cm]
    \item $\omega_{1}(p)=\omega_{2}(p)=1$ for every $p\in M\setminus B(p_0,\delta)$, \\[-0.3cm]
    \item $\displaystyle\int_M\omega_1^{-1}\,d\mu=1$,
     $\displaystyle \int_M\omega_2^{-1}\,d\mu=\infty$,
    \end{enumerate}
    where $B(p_0,\delta)$ is a ball of radius $\delta$ around $p_0$.
    Let $\Phi_1$ and $\Phi_2$ be the time transformations of $\{\eta^t\}_{t\in\bR}$ with reparameterization functions $r_k$ given by  $r_k(p)\eqdef\omega_k(p)^{-1}$ for every $p\ne p_0$, $k=1$, $2$, respectively. The flow $\Phi_1$ preserves the probability measure $\mu_1$ given by $d\mu_1=\omega_1^{-1}\,d\mu$, and we have
\begin{eqnarray}
	{r_1}_+(p)= 1&& \,\,
	{\rm  for }\,\,\mu{\rm -almost}\,\,{\rm every }\,\,p\,\,
	{\rm and }\,\,\mu_1{\rm-almost}\,\,{\rm every }\,\,p.
\end{eqnarray}
The flow $\Phi_2$ preserves the $\sigma$-finite measure $\mu_2$ given by $d\mu_2=\omega_2^{-1}\,d\mu$, and we have
\begin{eqnarray}
	{r_2}_+(p)= \infty&& \,\,
	{\rm for }\,\,\mu{\rm-almost}\,\,{\rm every}\,\,p\,\,
	{\rm  and }\,\,\mu_2{\rm -almost}\,\,{\rm every }\,\,p.\label{borel}
\end{eqnarray}
We sketch how property (\ref{borel}) can be shown. Choose some monotonically decreasing sequence of continuous functions
 $\eta_n\colon M\to (0,1]$ satisfying $\eta_n(p)=\omega_2(p)$ for every $p\in M\setminus B(p_0,2^{-n})$, $\eta_n(p_0)\le 2^{-n}$, and $\eta_n(p)\ge\omega_2(p)$ otherwise. In particular, $\eta_n$ converges  pointwise to $\omega_2$. Clearly, we have  $\xi_n(\mu)\eqdef \int_M\eta_n^{-1}\,d\mu<\infty$ for every $n\ge 1$ and $\lim_{n\to\infty}\xi_n(\mu) =\infty$. By the  Birkhoff ergodic theorem, for each $n$ there is a set $A_n\subset M$ of full $\mu$-measure for which ${\xi_n}_+(p)=\xi_n(\mu)$ and hence  ${r_2}_+(p)\ge\xi_n(\mu)$ for every $p\in A_n$. Using a Borel-Cantelli argument, we conclude that $(\eta_n^{-1})_+(p)=\infty$ for $\mu$-almost every $p$ and hence for $\mu_2$-almost every $p$.
From the arguments above, we obtain
\[
\lambda^{r_1}_k(p)= \lambda_k(p)
\,\,{\rm for }\,\,\mu_1{\rm-a. e. }\,\,p
,\text{ and }\,
{\lambda^{r_2}_k}(p)=0
\,\,{\rm for }\,\,\mu_2{\rm-a. e. }\,\,p.
\]
This construction demonstrates that an appropriate time transformation (here by means of $r_2$) can eliminate any non-zero Lyapunov exponents almost everywhere.
 }\end{example}

Recall that a set $A\subset M$ is of \emph{total probability} with respect to $\cM_{\rm e}(\Phi)$ if for all $\mu\in\cM_{\rm e}(\Phi)$ we have $\mu(A)=1$.

\begin{example}[Synchronized expansions]\label{ex:LyapPar}{\rm
		In the context of  Parry's
analysis of the so-called ``synchronization" of canonical measures~\cite{Par:86},
let us consider the example of a vector field with a  topologically transitive $C^1$ Anosov flow $\Phi$ on a compact smooth manifold $M$.
Extending the result in~\cite{Par:86}, we show that not only the measure of maximal entropy and the SRB measure can be synchronized (see Example~\ref{ex:syncb} below), but also the expansion rates for almost every point.
If $\Phi$ is $C^2$, then $x\mapsto E^u_x$ is H\"older continuous so that $D_x\varphi^t|_{E^u_x}$ and the function defined by
\begin{equation}\label{parry1}
	r(x)=\xi^u(x)
	\eqdef -\lim_{t\to0}\frac 1 t\log\,|{\rm det}\, D_x\varphi^t|_{E^u_x}|
\end{equation}
is
H\"older continuous.  If we assume that $\Phi$ is $C^2$ and also assume\footnote{This is a significant restriction made in order to obtain an orbit equivalent  $C^1$ flow. For a related discussion, see~\cite[Section~4]{Par:86}.} that the distribution $x\mapsto E^u_x$ is $C^1$, then $\xi^u$ is $C^1$ smooth and the  so obtained reparameterized flow $\Phi_r$ is a $C^1$ flow and orbit equivalent to $\Phi$.
It follows from the Birkhoff theorem and the multiplicative ergodic theorem that for any $\mu\in\cM_{\rm e}(\Phi)$, a $\mu$-full measure set of points satisfies
\[
r_+(x) = \lim_{t\to\infty}\frac{1}{t}\log\,\left|{\rm det}\, D_x\varphi^t|_{E^u_x}\right|
	= \sum_{i\colon\lambda_i(x)>0}\lambda_i(x).
\]
An immediate consequence of Theorem~\ref{theo.lyap} is that we have
$\sum_{\lambda_i^r(x)>0}\lambda_i^r(x) = 1$
for points $x$ in a set of total probability.
Therefore, for a transformation defined by (\ref{parry1}) there is synchronization in the sense
that $\int_{\mu} \sum_{\lambda_i(x)>0}\lambda_i(x)\,d\mu$ (before the time transformation)
is $\mu$-dependent while
$\int_{\mu_r} \sum_{\lambda^r_i(x)>0}\lambda^r_i(x)\,d\mu_r$ (after the time transformation) is
the same for all $\mu_r\in\cM_{\rm e}(\Phi_r)$.
In the particular case in which $E^u$ is one-dimensional, the set of points of total probability satisfies
\[
r_+(x) = \lim_{t\to\infty}\frac{1}{t}\log\,|| D_x\varphi^t|_{E^u_x}|| = \lambda_{\dim M}(x),
\]
and hence  $\lambda_{\dim M}^r(x)=1$
for points $x$ in a set of total probability.
}\end{example}

\subsection{Topological pressure and entropy}

In this section we focus on the dynamics on some compact
invariant set
$\Lambda\subset M$.
We continue to use $\cM(\Phi)$ to now denote
the set of $\Phi$-invariant probability measure supported on $\Lambda$.

In order to establish our main results,
we first briefly recall some concepts from the thermodynamic formalism.
Given $t>0$ and $\varepsilon>0$, we say that a set $E\subset \Lambda$ of points
is $(t,\varepsilon)$-separated if $x_1$, $x_2\in E$, $x_1\ne x_2$ implies that
$\rho(\varphi^s(x_1),\varphi^s(x_2))>\varepsilon$ for some $s\in[0,t]$.
Given a continuous function $\xi\colon \Lambda\to\bR$, the  \emph{topological pressure} of $\xi$ with respect to
$\Phi|_\Lambda$ is defined by
\begin{equation}\label{toppress}
P_{\rm top}(\Phi,\xi) \eqdef
\lim_{\varepsilon\to 0}\limsup_{t\to\infty}\frac{1}{t}\log
\sup\left\{\sum_{x\in E}\exp\int_0^t\xi(\varphi^s(x))\,ds\right\},
\end{equation}
where the supremum is taken over all $(t,\varepsilon)$-separated sets
$E\subset \Lambda$.
 We call $h_{\rm top}(\Phi) \eqdef P_{\rm top}(\Phi,0)$ the
\emph{topological entropy} of $\Phi|_\Lambda$. Note that $P_{\rm
  top}(\Phi,\xi)=P_{\rm top}(\varphi^1,\xi)$, where the latter  is the classically defined topological pressure of $\xi$ with respect to the map $\varphi^1$
\cite{Wal:81}.
We have the following variational principle
\begin{equation}\label{varprinc}
P_{\rm top}(\Phi,\xi)=
\sup_{\mu\in \cM_{\rm e}(\Phi)} \left(h_\mu(\varphi^1)+\int_\Lambda \xi \,d\mu\right),
\end{equation}
where $h_\mu(\varphi^1)$ denotes the metric entropy of $\mu$ with respect
to the map $\varphi^1$.
A measure that realizes the supremum in (\ref{varprinc}) is called an \emph{equilibrium state} for $\xi$ with respect to $\Phi$.
By the Abramov formula~\cite{Abr:66},
for every $t\in\bR$, we have
\begin{equation}\label{abra}
  h_\mu(\varphi^t) = | t|\, h_\mu(\varphi^1).
\end{equation}
One calls $h_\mu(\Phi)\eqdef h_\mu(\varphi^1)$ the \emph{metric entropy} of
$\mu$ with respect to the \emph{flow} $\Phi$.
For $\mu\in\cM(\Phi)$ and a flow $\Phi_r$ that is a time reparameterization of $\Phi$
by means of a function $r\in L^1(M,\mu)$ with $\int_\Lambda r\,d\mu>0$,
it is well known~\cite{Abr:66,Tot:66} that  the metric entropy of the measure $\mu_ r\in\cM(\Phi_r)$ satisfies
  \begin{equation}\label{sonneen}
  h_{\mu_ r}(\Phi_r)\, \int_\Lambda r\,d\mu \le h_\mu(\Phi).
  \end{equation}
  The equality holds if $\mu$ is ergodic.

As the pressure is defined for \emph{continuous} potentials only, we now specialize to time transformations with the same property. The following relation is an immediate consequence of the variational principle~(\ref{varprinc}) and
equations~(\ref{sonneen}) and~(\ref{defahh}).

\begin{theorem}\label{ute}
  Given a flow $\Phi_r$ that is a time reparameterization of $\Phi$ with an associated positive continuous  function $r\colon \Lambda\to \bR_{>0}$, for every continuous function $\xi\colon \Lambda\to \bR$ we have
  \begin{equation}\label{eq:pressoi}
  P_{\rm top}\Big(\Phi_r,
    \frac{\xi-P_{\rm top}(\Phi,\xi)}{ r}\Big)=0.
  \end{equation}
\end{theorem}

The following discussion sheds new light on synchronizations of Lyapunov exponents discussed in Example~\ref{ex:LyapPar}.

\begin{example}[Synchronized equilibrium states]{\rm
\label{ex:syncb}
Let $\Phi$ be a $C^2$ topologically transitive Anosov flow on $M$.
Recall that for every $q\in \bR$, there exists a unique equilibrium state $\mu_q$ for the potential $x\mapsto q\,\xi^u(x)$. In particular, $\mu_0$ is the unique measure of maximal entropy and $\mu_1$ is the SRB measure for $\Phi$.
Let us consider $\xi^u(x,t)=-\log \,| {\rm det}\,D_x\varphi^t|_{E^u}|$ and denote by $\xi^u_r(x,\tau)$ the corresponding function for the flow $\Phi_r$ with respect to the unstable distribution $E^u_r$ of $\Phi_r$. It can be shown that
\[
\xi^u_r(x,\tau) = \xi^u(x,t(x,\tau)) + \eta(\varphi^\tau_r(x))-\eta(x)
\]
for some continuous function $\eta$ (see~\cite{AnoSin:67} or~\cite[Theorem~1]{Par:86}).
With the particular choice of time reparameterization $r(x)=\xi^u(x)$ in~(\ref{parry1}), we obtain
\[
\xi^u_r(x,\tau) = \tau + \eta(\varphi^\tau_r(x))-\eta(x),
\]
which implies that $\xi^u_r$ defined by $\xi^u_r(x)=\lim_{\tau\to0}\xi^u_r(x,\tau)/\tau$ is $\Phi_r$-cohomolo\-gous to $1$.%
\footnote{Recall that $\xi\colon M\to\bR$ is $\Phi$-cohomologous to a constant $c\in\bR$ if there exists a bounded measurable function $\eta\colon M \to\bR$ such that $\xi(x)-c=\lim_{t\to 0}(\eta(\varphi^t(x))-\eta(x))/t$ for every $x\in M$. }
This means that for any $q\in\bR$, the potential $q\,\xi^u_r$ is cohomologous to a constant, and hence there is a unique measure  that is the equilibrium state for every $q\,\xi^u_r$. More important, this measure coincides with the measure of maximal entropy and the SRB measure for $\Phi_r$. It is an immediate consequence that
\[
P_{\rm top}(\Phi_r,q\,\xi^u_r) = 1-q\quad {\rm for }\,\,{\rm every }\,\,q\in\bR.
\]
For a comparison, consider the nontrivial case in which $\xi^u$ is not $\Phi$-cohomo\-lo\-gous to a constant or, equivalently, that $\mu_0\ne\mu_1$. A change of time  parameterization transforms the $\Phi$-SRB measure $\mu_1$ into the $\Phi_r$-SRB measure.
Observe that the measure $\mu_0$ of maximal entropy with respect to $\Phi$ then must satisfy $h_{\mu_0}(\Phi)<\int_Mr\,d\mu_0$. This implies that the measure $\mu_{0r}$
satisfies $h_{\mu_{0r}}(\Phi_r)<1=h_{\rm top}(\Phi_r)$, that is, after the change of time the transformed measure is no longer of maximal entropy.
}\end{example}

In the following theorem we put the above example in a more general framework, which holds true for arbitrary potentials.
Recall that a set $\Lambda\subset M$ is said to be a \emph{basic set} of an axiom A flow $\Phi$ if $\Lambda$ contains a dense  set of periodic orbits and $\Phi|_\Lambda$ is hyperbolic, locally maximal, and topologically transitive.

\begin{theorem}
	Let $\Lambda$ be a basic set of a $C^1$ axiom A flow $\Phi$ and $\xi\colon M\to\bR_{>0}$ be a $C^1$ function.
	Given the flow $\Phi_r$ that is the time reparameterization of $\Phi$ with the associated function $r=\xi$, we have
	\begin{equation}\label{eq:prssr}
		P_{\rm top}(\Phi_\xi,q\,\xi) = q_0 - q\quad{\rm for}\,\,{\rm every }\,\,q\in\bR,
	\end{equation}
	where $q_0$ is the unique number satisfying $P_{\rm top}(\Phi,q_0\,\xi)=0$. Moreover, the equilibrium state for $q\,\xi$ with respect to $\Phi_\xi$ coincides with the measure of maximal entropy with respect to $\Phi_\xi$ for every $q\in\bR$.
\end{theorem}

\begin{proof}
By a well-known result on the equidistribution of closed orbits for $\Phi$  we know that the unique equilibrium state $\mu(\xi)$ of $\xi$ with respect to $\Phi$ is obtained as the weak$\ast$ limit of the weighted orbital measures: given any $\varepsilon>0$,
\[
\frac{\sum_\gamma \ell(\gamma)e^{\xi(\gamma)} m_\gamma}
	{\sum_\gamma \ell(\gamma)e^{\xi(\gamma)} } \to
	\mu(\xi)\quad {\rm for }\quad T\to\infty,
	{\rm where }\,\,\xi(\gamma)\eqdef \int_\Lambda\xi\,dm_\gamma,
\]
with summation taken over closed orbits $\gamma$ with period $\ell(\gamma)$ between $T$ and $T+\varepsilon$ and $m_\gamma$ denoting the invariant probability measure supported on $\gamma$ (see~\cite{PolShaTunWal:08}). Considering the change of time parameterization given by $r=\xi$, any closed orbit $\gamma_r=\{\varphi^\tau_r(x)\colon 0\le\tau\le \ell(\gamma_r)\}$ with respect to $\Phi_r$ satisfies
$ \int_0^{\ell(\gamma_r)} \xi(\varphi^\tau_r(x))\,d\tau
 =\ell(\gamma)$.
Hence, for any real parameter $q$, the equilibrium state $\mu_r(q\,\xi)$ of $q\,\xi$ with respect to $\Phi_r$ is obtained as the limiting measure
\[
\frac{\sum_{\gamma_r}  \ell(\gamma_r)\,m_{\gamma_r}}
	{\sum_{\gamma_r} \ell(\gamma_r)} \to
	\mu_r(q\,\xi)\quad{\rm for }\quad T\to\infty.
\]
In particular, this equilibrium state coincides with the measure of maximal entropy for $\Phi_r$. Combined with (\ref{eq:pressoi}), this leads to $h_{\rm top}(\Phi_r)=q_0$, which implies (\ref{eq:prssr}) and proves the theorem.
\end{proof}

Now we present some results on topological entropy. Theorem~\ref{ute} leads to the following characterization of the topological pressures in terms of the topological entropy.

\begin{theorem}
  Given a flow $\Phi_r$ that is a time reparameterization of $\Phi$ with an associated positive continuous function $r\colon \Lambda\to \bR_{>0}$,  we have
  \[
   P_{\rm top}\Big(\Phi_r,-\frac{h_{\rm top}(\Phi)}{ r}\Big)=0.
  \]
\end{theorem}

\begin{proof}
  It follows from the definition (\ref{toppress}) that $P_{\rm top}(\Phi,-h_{\rm top}(\Phi))=0$ and that  the number $h_{\rm top}(\Phi)$ is the only zero of the function $s\mapsto P_{\rm top}(\Phi,-s)$ on the real line. The statement then follows from Theorem~\ref{ute} with $\xi=0$.
\end{proof}

If the reparameterization function $ r\colon \Lambda\to\bR_{>0}$ is  positive, by the variational principle (\ref{varprinc}) and by means of~(\ref{sonneen}) we can characterize the topological entropy of the flow
$\Phi_r$ as
\[
h_{\rm top}(\Phi_r) = \sup_{\mu\in\cM_{\rm e}(\Phi)}
\frac{h_\mu(\Phi)}{\int_\Lambda r \,d\mu} \,.
\]
The topological entropy is then bounded as
\[
\frac{h_{\rm top}(\Phi)}{\sup_{\mu\in\cM(\Phi)}  \int_\Lambda r\,d\mu}\le
h_{\rm top}(\Phi_r) \le
h_{\rm top}(\Phi) \sup_{\nu\in\cM(\Phi_r)} \int_\Lambda\frac{1}{r}\,d\nu \,.
\]
In general the topological entropy can change with the time transformation of
the flow, as the following examples demonstrate.

\begin{example}[Topological entropy]{\rm
Let us consider the example of a vector field $f$ with a topologically transitive Anosov flow $\Phi$ on a compact smooth manifold $M$.
Considering a time transformation by means of a $C^1$ function $r\colon M\to\bR_{>0}$, we obtain
\[
\frac{h_{\rm top}(\Phi)}{\int_M r\,d\mu_\Phi}\le
h_{\rm top}(\Phi_{r})\le
h_{\rm top}(\Phi)\int_M\frac{1}{r}\,d\mu_{\Phi_r},
\]
where $\mu_\Phi$ and $\mu_{\Phi_r}$ denote the Margulis measures~\cite{KatKniWei:91} with respect to $\Phi$ and to $\Phi_r$ (i.e., the unique and hence ergodic measure of maximal entropy).
In particular, the topological entropy changes when $r(x)\equiv r_0\ne 1$.
}\end{example}

Zero  topological entropy is an invariant for orbit
equivalent flows without fixed points~\cite{Ohn:80}.  In equivalent flows
\emph{with} fixed points, this property can change with a time reparameterization
if we consider a  more general, not integrable,
reparameterization function.

\begin{example}[Entropies for non-integrable  $r$]\label{ex:Sun}{\rm
    Sun\,\emph{et\,al.}~\cite{SunYouZho:06} constructed $C^\infty$ flows $\Phi_1$ and $\Phi_2$ on a
    compact Riemannian manifold $M$ of dimension greater than $2$ that possess one
    fixed point $x_0\in M$ and that are time transformations of each other.  The flow $\Phi_1$ preserves an ergodic  probability measure $\mu$ satisfying $h_\mu(\Phi_1)>0$ and hence we have $h_{\rm top}(\Phi_1)>0$.
    The only  ergodic probability measure $\nu$ preserved by $\Phi_2$ is supported on the  fixed point,  leading to $h_\nu(\Phi_2)=h_{\rm top}(\Phi_2)=0$.
  In view of (\ref{sonneen}) and $h_\mu(\Phi_1)>0$ it can be concluded that the orbit equivalent flows $\Phi_1$ and $\Phi_2$ are related to each other by a time transformation with a function  that is not $\mu$-integrable.
  We refer to~\cite{SunYouZho:06} for details on the construction of the flow.
}\end{example}

\subsection{Recurrences and waiting times}
The Poincar\'e recurrence theorem says that, given a measure $\mu\in\cM(\Phi)$,
almost all orbits starting from a set of positive measure will come back to it
infinitely many times (see, e.g.,~\cite{CorFomSin:82,Wal:81}).
Quantitative indicators of such recurrent behavior are given by
the lower and the upper recurrence rates with respect to the flow $\Phi$. Given $x$, $y\in M$, let  us define the \emph{lower} and \emph{upper waiting time indicators} by
\[
\underline R(\Phi,x,y) \eqdef
\liminf_{\varepsilon\to 0}
\frac{\log E_{B(x,\varepsilon)}(\Phi,y)}{-\log\varepsilon}, \,\,
\overline R(\Phi,x,y) \eqdef
\limsup_{\varepsilon\to 0}
\frac{\log E_{B(x,\varepsilon)}(\Phi,y)}{-\log\varepsilon},
\]
where $B(x,\varepsilon)\eqdef\{z\colon\rho(x,z)<\varepsilon\}$ and
\[
E_A(\Phi,y)\eqdef \inf\left\{t>e_A(\Phi,y) \colon \varphi^t(y)\in A\right\}
\]
denotes the \emph{time of first (re-)entrance} of the trajectory of $y\in
\Lambda$ into $A$. Here $e_A(\Phi,y)\eqdef\inf\{t>0\colon\varphi^t(y)\notin
A\}$ denotes the \emph{escape time} of $y$ from $A$. If
$E_{B(x,\varepsilon)}(\Phi,y)$ is infinite for some $\varepsilon$, then
$\underline R(\Phi,x,y)$ and $\overline R(\Phi,x,y)$ are set to be infinite.
For $x=y$ we call $\underline R(\Phi,x)\eqdef \underline R(\Phi,x,x)$ and  $\overline R(\Phi,x)\eqdef \overline R(\Phi,x,x)$ the \emph{lower} and the \emph{upper recurrence rates} at $x$ with respect to $\Phi$, respectively.
Given a measure $\mu\in\cM(\Phi)$ and a positive measure set $A$,
it follows from the Poincar\'e recurrence theorem that $E_A(\Phi,x)$ is  finite for $\mu$-almost every  $x\in A$.

We study transformations of recurrence and waiting rates under a change of time
parameterization by studying their local behavior.
Notice that for $x\in\widehat M$,
\begin{equation}\label{godo}
  E_{B(x,\varepsilon)}(\Phi_r,y)=
  \int_0^{E_{B(x,\varepsilon)}(\Phi,y)}  r(\varphi^s(y))\,ds,
\end{equation}
which follows immediately from the time transformation (\ref{defah}).

\begin{lemma}\label{ddprop}
Given $\mu\in\cM(\Phi)$, for
$\mu$-almost every $x\in \widehat M$ we have
\begin{equation}\label{holie}
\lim_{\varepsilon\to 0}
\frac{E_{B(x,\varepsilon)}(\Phi_r,x)}{E_{B(x,\varepsilon)}(\Phi,x)} =
r_+(x).
\end{equation}
\end{lemma}

\begin{proof}
	A well-known recurrence result~\cite[p.\,61]{Fur:81} states that for $\mu$-almost every $x\in M$ we have $\liminf_{t\to\infty}\rho(x,\varphi^t(x))=0$. If $x$ is non-periodic, this implies that $E_{B(x,\varepsilon)}(\Phi,x)\to \infty$ as $\varepsilon\to 0$, and (\ref{holie}) follows from the Birkhoff ergodic theorem. If $x$ is periodic with period $T>0$, then $E_{B(x,\varepsilon)}(\Phi,x)=T$, and (\ref{godo}) implies (\ref{holie}).
\end{proof}

We call $\mu\in\cM(\Phi)$ \emph{non-atomic} if it is not supported on a periodic orbit.

\begin{lemma}\label{ddprop}
Let $\mu\in\cM(\Phi)$ be non-atomic. Then
given $x\in M$, for $\mu$-almost every $y\in \widehat M$ we have
\begin{equation}\label{holiholiu}
\lim_{\varepsilon\to 0}
\frac{E_{B(x,\varepsilon)}(\Phi_r,y)}{E_{B(x,\varepsilon)}(\Phi,y)} =
r_+(y).
\end{equation}
\end{lemma}

\begin{proof}
Let us assume that there is a set $A\subset \widehat M$
of positive measure such that $E_{B(x,\varepsilon)}(\Phi,y)\not\to\infty$ when $\varepsilon\to0$  and is bounded, that is, there exists a  positive number $C$ such that for every $y\in A$ we have
$\limsup_{\varepsilon\to 0}E_{B(x,\varepsilon)}(\Phi,y)<C$.
Since the measure is nonatomic, let us consider a number $\varepsilon$ small enough such that $\mu(B(x,\varepsilon))<C^{-1}\mu(A)$.  Since every $y\in A$ must enter $B(x,\varepsilon)$ in less than $C$ time units, there is a $t\le C$ and a subset $A' \subset A$  such that $\mu(A') \geq C^{-1}\mu(A)$  and $\varphi^t(A')\subset B(x,\varepsilon)$.
Because $\varphi^{-t}(B(x,\varepsilon))$ will have larger measure than $B(x,\varepsilon)$, this contradicts the invariance of the measure $\mu$ and hence proves that $A$ has zero measure.
Accordingly, for $\mu$-almost every $y$ we have $E_{B(x,\varepsilon)}(\Phi,y)\not\to\infty$ as $\varepsilon\to0$ and, in view of  (\ref{godo}), this implies (\ref{holiholiu}).
\end{proof}

Thus, we obtain the following result.

\begin{theorem}\label{theo:rec}
	Given $\mu\in\cM(\Phi)$ and a flow $\Phi_r$ that is a time reparameterization of $\Phi$ with an associated function $r\in L^1(M,\mu)$,  for $\mu$-almost every $x$  we have $\underline R(\Phi_r,x)=\underline R(\Phi,x)$  and $\overline R(\Phi_r,x)=\overline R(\Phi,x)$. If $\mu$ is non-atomic,  then for $\mu$-almost every $y$  we have $\underline R(\Phi_r,x,y)=\underline R(\Phi,x,y)$  and $\overline R(\Phi_r,x,y)=\overline R(\Phi,x,y)$.
\end{theorem}

Waiting times may change after a change of time parameterization if for a ``typical point" $y$ we have $r_+(y)=\infty$. A construction of such a flow and reparameterization function can
be obtained following Example~\ref{ex:suspen} or the model in~\cite{SunYouZho:06}. We illustrate this fact instead with a simple example in which the measure is atomic.

\begin{example}[Waiting times for singular transformations]{\rm
    Consider  the flow $\Phi$ generated by the vector field $f(x)=-x$, for $x\in\bR$.
    In the case of the time reparameterized flow $\Phi_{r_1}$ with the associated
    function $r_1(x)=-\log \,| x|$, for $x=0$ and for every $y\in\bR\setminus\{0\}$ we have ${r_1}_+(y)=\infty$ and
    \[
    R(\Phi_{r_1},0,y)=0<R(\Phi,0,y)=1.
    \]
    In the case of the time reparameterized semi-flow $\Phi_{r_2}=\{\varphi_{r_2}^t\}_{t\ge 0}$ with the associated function $r_2(x)=| x|^{-c}$, $c\ge 1$, for $x=0$ and every $y\in\bR\setminus\{0\}$ we have ${r_2}_+(y)=\infty$ and
    \[
    R(\Phi,0,y)=1\le R(\Phi_{r_2},0,y)=c.
    \]
This shows that $r_+(y)<\infty$
is necessary for the invariance of the waiting time statistics, which is satisfied for $\mu$-almost every $y$  provided $\mu\in\cM(\Phi)$ and $r\in L^1(M,\mu)$.
}
\end{example}

\section{Transformation of dimensions of measures}\label{sec:6}

We now consider concepts from dimension theory for dynamical
systems. We refer to~\cite{Pes:98} and references therein
for the proofs of the known properties used below.

\subsection{Hausdorff dimension and local dimension}\label{sec:6.1}

Given a $\sigma$-finite Borel measure $\mu$ and $x\in M$, we define the \emph{lower local dimension} of $\mu$ at $x$ by
\begin{equation}\label{defptwdim}
\underline d_\mu(x) \eqdef
\liminf_{\varepsilon\to 0} \frac{\log \mu(B(x,\varepsilon))}{\log\varepsilon}.
\end{equation}
Analogously, the \emph{upper local dimension} $\overline d_\mu(x)$ of $\mu$ at
$x$ is defined by replacing the lower limit with the upper limit. If
$\underline d_\mu(x) = \overline d_\mu(x)\eqdef d_\mu(x)$, we call the common
value the \emph{local dimension} of $\mu$ at $x$.
If $\mu$ is a Borel probability measure, we consider the \emph{Hausdorff dimension} of $\mu$ defined by
\[
\dim_{\rm H}\mu
  \eqdef \inf\{\dim_{\rm H}A\colon A\subset M \,\,{\rm and }\,\,\mu(A)=1\},
\]
where $\dim_{\rm H}A$ denotes the Hausdorff dimension of a set $A$.
Note that
\begin{equation}\label{hig}
	\dim_{\rm H}\mu = \essup_{x\in M}\underline d_\mu(x) ,
\end{equation}
where $\essup$ denotes the essential supremum with respect to $\mu$.

\begin{theorem}\label{cor.dim1}
  For any $\mu\in\cM(\Phi)$ and $\widehat \mu_ r$ defined by $d\widehat\mu_r=r\,d\mu$, we have
  $\underline d_{\widehat\mu_ r}(x)=\underline d_\mu(x)$,
  $\overline d_{\widehat\mu_ r}(x)=\overline d_\mu(x)
  $
  for all $x\in\widehat M$.
  Moreover, if $r\in L^1(M,\mu)$ and $\mu_r$ is given by~(\ref{sua}), we have $\dim_{\rm H}\mu_r \le \dim_{\rm H} \mu$.
\end{theorem}

\begin{proof}
	Recall that $\Phi_r$ is  $C^{\ell}$.
	For every $x\in\widehat M$, the function $r$ is $C^{\ell-1}$ and positive on $B(x,\varepsilon)$ for sufficiently small $\varepsilon$, and hence the first two equalities are immediate.  Note that for every $\Phi$-invariant $A\subset M$ we have $\mu_r(A)=0$ whenever  $\mu(A)=0$. This implies $\dim_{\rm H}\mu_ r \le\dim_{\rm H}\mu$.
\end{proof}

\begin{example}[Hausdorff dimension for non-integrable $r$]{\rm
	In the context of Example~\ref {ex:suspen}, we have
$d_\mu(p)=3=\dim_{\rm H}\mu$ for every $p\in M$. The measure $\mu_2$ defined by $d\mu_2=\omega_2^{-1}\,d\mu$ (induced by the time reparameterization in that example) is not normalizable. It satisfies $d_{\mu_2}(p)=3$ for every $p\in \widehat M$. It can be shown that the (only ergodic) $\Phi_2$-invariant probability measure $\nu$  is supported on the fixed point $p_0$ and hence satisfies $d_\nu(p_0)=d_{\nu}(p)=0=\dim_{\rm H}\nu$ for  every $p\in M$.
    }\end{example}

\subsection{Information dimension}

Given a Borel probability measure $\mu$, the \emph{lower information
  dimension} of $\mu$ is defined by
\[
\underline D_1(\mu) \eqdef \liminf_{\varepsilon\to
  0}\frac{1}{\log\varepsilon}\int_M
  \log\mu(B(x,\varepsilon))\,d\mu(x)
\]
(this definition is equivalent to the standard one using grids~\cite{BarGerTch:01}).
Analogously, the \emph{upper information dimension} $\overline D_1(\mu)$
is defined by replacing the lower limit with the upper limit.

\begin{theorem}\label{theo:inf}
  Given a flow $\Phi$ and a measure $\mu\in\cM(\Phi)$ satisfying
  \begin{equation}\label{ptdim}
    \underline d_{\mu}(\cdot) = \overline
    d_{\mu}(\cdot) \eqdef d_\mu(\cdot)
  \end{equation}
  $\mu$-almost everywhere, and given a flow
  $\Phi_r$ that is a time reparameterization of $\Phi$ with an associated function
  $r\in C^{\ell-1}$
    and $\mu_ r$ given by (\ref{sua}), we have
  \begin{equation}\label{hicks}
  D_1(\mu_ r) = \int_M\frac{ r_+(x)}{\int_M r \,d\mu}d_\mu(x)\,d\mu(x).
  \end{equation}
  Moreover,  if additionally
  \begin{enumerate}
    \item $\mu$ (and hence $\mu_ r$) is ergodic, or if
    \item $\underline d_\mu = \overline d_\mu=d$ almost everywhere for some
    constant $d$,
  \end{enumerate}
	then we have $D_1(\mu_ r)=D_1(\mu)$.
\end{theorem}

\begin{proof}
As a consequence of the Fatou lemma, we obtain
\begin{equation}\label{pow}
  \int_M\underline d_\mu(x)\,d\mu(x) \le
  \underline D_1(\mu) \le \overline D_1(\mu) \le
  \int_M\overline d_\mu(x)\,d\mu(x).
\end{equation}
Under the assumptions of the theorem, from (\ref{pow}), (\ref{sua}),
and Theorem~\ref{cor.dim1}, we have
  \[
  D_1(\mu_ r)
  = \int_{\widehat M}  d_{\mu_r}(x)\frac{ r(x)}{\int_M r \,d\mu}d\mu(x)
   = \int_{M} d_\mu(x)\frac{ r(x)}{\int_M r \,d\mu}d\mu(x).
  \]
  Since $\varphi^s\colon M\to M$, $s\in (0,1]$, is a Lipschitz
  map with Lipschitz inverse, one can verify that $d_\mu(x)=
  d_\mu(\varphi^s(x))$. Thus
  for arbitrary $s> 0$,
  \[
  \int_M d_\mu(x) r(x)\,d\mu(x)=
  \int_M d_\mu(x)\frac{1}{s} \tau(x,s) \,d\mu(x),
  \]
  from which we obtain~(\ref{hicks}).
  If $\mu$ is ergodic, then $ r_+(x)=\int_M r \,d\mu$
for almost every $x$,
  and   hence $D_1(\mu_r)=D_1(\mu)$.
If $\underline d_\mu=\overline d_\mu=d$ almost everywhere, then we can use the
relation $\int_M r_+\,d\mu=\int_M r \,d\mu$, which follows from the Birkhoff ergodic
theorem, to conclude that $D_1(\mu_ r)=D_1(\mu)=d$.
  \end{proof}

Theorem~\ref{theo:inf} implies that for a
non-ergodic measure $\mu$ with a local dimension that is not constant
on a set of positive measure, we can find a suitable change of the time
parameterization, and hence a ``redistribution" of the
invariant measure, such that the Hausdorff dimension and the information
dimension of the resulting measure $\mu_ r$ differ.
Thus, the information dimension can sensitively depend on the time
parameterization and as such cannot be regarded as an invariant.

We conclude this section by giving illustrative examples.\\[0.3cm]
\addtocounter{example}{1}
\noindent
\emph{Example~\theexample\,a
 (Noninvariance of $D_1$)}\textbf{. }{\rm
	Let $\Phi=\{\varphi^t\}_{t\in\bR}$ be a flow possessing a fixed point $x_0$ and a periodic point $y=\varphi^T(y)$ of period $T>0$. Let $\mu=\delta_{x_0}$ be the Dirac-$\delta$ measure supported on $x_0$ and let $\nu$ be the $\Phi$-invariant Borel probability measure that is supported on the periodic orbit $\OO$ through $y$. Then, $D_1(\mu)=0=\underline d_\mu(x_0)=\overline d_\mu(x_0)\eqdef d_\mu(x_0)$ and $D_1(\nu)=1=\underline d_\mu(x)=\overline d_\mu(x)\eqdef d_\mu(x)$ for every $x\in\OO$. Given $\alpha\in (0,1)$, the (non-ergodic) measure $\widetilde\mu\eqdef \alpha\nu+(1-\alpha)\mu$ satisfies~(\ref{ptdim}), and hence
\[
D_1(\widetilde\mu)=\alpha D_1(\nu) +(1-\alpha)D_1(\mu) =1-\alpha
>\frac{(1-\alpha)r_+(y)}{r(x_0)\alpha+(1-\alpha)r_+(y)}
=D_1(\widetilde \mu_r)
\]
whenever $\Phi_r$ is a flow obtained from $\Phi$ by means of a smooth change of the time
parameterization satisfying  $r(x_0)<r_+(y)$.
}
\\[0.3cm]\indent
Consider a flow on a compact Riemannian manifold that can be represented
as a suspension flow over a two-sided topological Markov chain with H\"older
continuous roof function. Such flows naturally occur as models for
flows that possess strong hyperbolic behavior. Primary examples are
geodesic flows on  compact Riemannian manifolds with negative sectional
curvature and their time reparameterizations.
Note that the property~(\ref{ptdim}) is satisfied for any invariant
Borel probability measure $\mu$~\cite[Theorem 1]{BarRadWol:04}, and that
$d_\mu$ is constant almost everywhere if $\mu$ is ergodic. This is the
starting point for the following (more involved) variation of Example~\theexample\,a.\\[0.3cm]
\noindent
\emph{Example~\theexample\,b (Noninvariance of $D_1$)}\textbf{. }{\rm
	Let in particular $M$ be a compact orientable Riemannian surface of class
$C^4$ with negative curvature $K$ and consider the geodesic flow
$\Phi=\{\varphi^t\}_{t\in\bR}$ on the unit tangent bundle $SM$. Note that the normalized Liouville measure $\mu_{SM}$ on $SM$, which is induced
by the volume $\mu_M$ on $M$, is ergodic (see~\cite{BarPes:05} for
details). Hence, for $\mu_{SM}$-almost every $x\in SM$ we have
\[
K_+(x) \eqdef \lim_{t\to\infty}\frac{1}{t}\int_0^tK(\varphi^s(x))\,ds
= \int_MK\,d\mu_M = 2\pi\chi_M,
\]
where $\chi_M$ denotes the Euler
characteristic of $M$ and where the last equality follows from the
Gau{\ss}-Bonnet theorem.
Let us further assume that $K$ is not $\Phi$-cohomologous to a
constant.
It follows from the Livshitz theorem  for flows (see, for
example,~\cite[Theorem 19.2.4]{KatHas:95}) that there exist periodic points
$x_1=\varphi^{t_1}(x_1)$ and $x_2=\varphi^{t_2}(x_2)$ such that
\[
\frac{1}{t_1}\int_0^{t_1}K(\varphi^s(x_1))\,ds
\ne \frac{1}{t_2}\int_0^{t_2}K(\varphi^s(x_2))\,ds,
\]
and, in particular, there exists a periodic point $x_0\in SM$ such that
$K_+(x_0)\ne 2\pi\chi_M$. We assume that
$k_0\eqdef K_+(x_0)<2\pi\chi_M$.
We consider the $\Phi$-invariant Borel probability measure $\nu$ that is
supported on the periodic orbit $\OO$ through $x_0$ and satisfies
$\underline d_\nu(x)= \overline d_\nu(x)  = D_1(\nu) = 1$ for every point $x$ on this
orbit.
Analogously, for $\mu_{SM}$-almost every $y\in SM$ we have $\underline d_{\mu_{SM}}(y)=
\overline d_{\mu_{SM}}(y) = D_1(\mu_{SM})  = 3$.
In addition, we have $K_+(x)= k_0<2\pi\chi_M=K_+(y)$
for $\nu$-almost every $x$ and $\mu_{SM}$-almost every $y$. Consider the positive function $r\colon M\to\bR_{>0}$ given by $r(x)=-K(x)$ for every $x\in SM$.
(Note that we can, in fact, take any other H\"older continuous function $\xi\colon SM\to\bR$ that is not $\Phi$-cohomologous to a constant and consider the function $r=\xi-\min_{x\in SM}\xi(x)+1$.) Given $\alpha\in (0,1)$, the measure $\widetilde\mu\eqdef
\alpha\nu+(1-\alpha)\mu_{SM}$ satisfies~(\ref{ptdim}) and hence
\[
D_1(\widetilde\mu)=\alpha D_1(\nu) +(1-\alpha)D_1(\mu_{SM})
= 3-2\alpha.
\]
If we now consider the flow $\Phi_r$ obtained from $\Phi$ after a change of
time parameterization with the function $r=-K$, we obtain
\[
D_1(\widetilde\mu_r)=
\frac{\alpha k_0 }{\alpha k_0 + (1-\alpha)2\pi\chi_M }D_1(\nu) +
\frac{(1-\alpha)2\pi\chi_M }{\alpha k_0 + (1-\alpha)2\pi\chi_M}D_1(\mu_{SM}),
\]
where $k_0<2\pi\chi_M$ implies $\alpha k_0/[\alpha k_0+(1-\alpha)2\pi\chi_M]>\alpha$, and hence  $D_1(\widetilde\mu)<D_1(\widetilde\mu_r)$.
}

\subsection{Generalized dimensions}

Given $q> 0$, $q\ne 1$, we define the \emph{lower generalized dimension of
  order $q$} of a Borel probability measure $\mu$ on $M$ by
\[
\underline D_q(\mu) \eqdef
\frac{1}{q-1}\liminf_{\varepsilon\to 0} \frac{1}{\log\varepsilon}
\log\int_M\mu(B(x,\varepsilon))^{q-1}\,d\mu(x).
\]
The functions $\overline D_q(\mu)$ are defined analogously with the upper
limit and are called the \emph{upper generalized dimension} of order $q$ of
$\mu$.

Although the generalized dimensions \emph{a priori} do not involve any
dynamics, they were introduced to obtain  dynamical information
 by observing individual trajectories~\cite{GraPro:83}, thus admitting
dynamical interpretations when $\mu$ is an invariant measure of the system.
Of particular dynamical interest is the so-called \emph{lower}
(\emph{upper}) \emph{correlation dimension} $\underline
D_2(\mu)$ ($\overline D_2(\mu)$), which is the most accessible one in numerical
computations based on time series analysis (see~\cite{Pes:93,Pes:98}  and references therein).

The discrete analogs of the above defined dimensions, often considered
in the physics literature, are called \emph{generalized lower (upper) R\'enyi
  dimensions} of order $q$.
Given $\varepsilon>0$, cover the support of $\mu$ with boxes $B_k$ of a
uniform grid of size $\varepsilon$ and denote by $N(\varepsilon)$ the
number of boxes needed to cover this set. For $q> 0$, $q\ne 1$, let
\begin{equation}\label{neq49}
\underline{RD}_q(\mu) \eqdef
\frac{1}{q-1}\liminf_{\varepsilon\to 0}\frac{1}{\log\varepsilon}
\log\sum_{k=1}^{N(\varepsilon)} \mu(B_k)^q
\end{equation}
and define $\overline{RD}_q(\mu)$ analogously  by replacing the lower limit with
the upper limit. These quantities are invariant under
smooth transformations of the phase space and their definitions
are independent of the choice of the grid. Moreover, $\underline
{RD}_q(\mu)=\underline D_q(\mu)$ and $\overline {RD}_q(\mu)=\overline D_q(\mu)$
(see~\cite{Pes:98,BarGerTch:01} and references therein). Furthermore, we now show that the
generalized dimensions of flow-invariant measures do not depend on a
particular choice of time parameterization.

\begin{theorem}\label{theo:HP}
Given a flow $\Phi_r$ that is a time reparameterization of $\Phi$ with an associated uniformly bounded positive function $r\colon M\to \bR_{>0}$,  for $\mu\in\cM(\Phi)$ and $\mu_ r$ given by~(\ref{sua}) we have $\underline D_q(\mu_ r) = \underline D_q(\mu)$ and $\overline D_q(\mu_ r)=\overline D_q(\mu)$ for every $q> 0$, $q\ne 1$.
\end{theorem}

\begin{proof}
  It suffices to notice that, by equation~(\ref{sua}), measures $\mu$ and $\mu_ r$ are absolutely
  continuous with respect to each other. Since $ r$ is
  positive
  and bounded, there exists positive constants $c_1$, $c_2$ such
  that $c_1\mu(B(x,\varepsilon))\le \mu_ r(B(x,\varepsilon))\le
  c_2\mu(B(x,\varepsilon))$ for every $x$ and $\varepsilon>0$.
  This applied to the definitions of $\underline D_q$ and $\overline D_q$
  proves the statement.
\end{proof}

\begin{remark}{\rm
We emphasize that the Examples~\theexample\,a,\,b illustrate the non-invariance of the
information dimension already in the case of orbit equivalent flows with a differentiable reparameterization function $r$. In contrast, note that, by Theorem~\ref{theo:HP}, in each of these examples we have $D_q(\mu)=D_q(\mu_r)$ for every $q> 0$, $q\ne1$.
The implications of the non-invariance of the information dimension for the Kaplan-Yorke relation are discussed in~\cite{MotGel:09}.
}\end{remark}

\section{Concluding remarks}

The transformation of dynamical quantities established in this paper addresses
a longstanding problem in general relativity and cosmology, namely,
of whether chaos is a property of the physical system or a property of the coordinate
system~\cite{HobBurCol:94}.  At the center of this discussion is the mixmaster
cosmological model~\cite{misner69},  a spatially homogeneous anisotropic solution of
Einstein's equations that has been conjectured to describe the dynamics of the early
universe. This model can be described as a geodesic flow on a Riemannian
manifold with negative curvature~\cite{chiltre72}, which nevertheless has been shown
to have  positive or vanishing largest Lyapunov exponent depending on the time
coordinate  adopted~\cite{FraMat:88}. This problem has resisted rigorous
solution because, as  shown in \cite{Mot:03} and~\cite{MotSaa:09}, the transformations used in cosmology
are not guaranteed to preserve the normalization property of the invariant measure
and because the identification of truly invariant indicators of chaos is often elusive.
To this regard, our results show once and for all that Lyapunov exponents, entropies,
and dimension-like characteristics can be used to make invariant assertions about
chaos. However, the same results also show that the {\it values} of some quantities
that have been previously conjectured to be invariant, such as the information dimension
and topological entropy  (see~\cite{MotLet:01} for a related discussion), are not invariant
in general.

Finally, we observe that there are several directions along which we expect this work
to be extended. One concerns applications, such as in the above mentioned study
of relativistic dynamics and to probe the invariance of dynamical quantities and identities,
as exemplified by our synchronization analysis. Another direction concerns the study of
different dynamical quantities and the relations between them. The spectrum of return
time dimension~\cite{HayLueManVai:02},
for example,  which relates the recurrence times to the multifractal properties of strange
attractors~\cite{Gal:05},  can be shown to be invariant under uniformly bounded time
transformations, whereas an exponential distribution of first return times can change within
exponential bounds.
Such transformation properties can stimulate further investigation on how robust the properties
of the return time statistics are, particularly when the system gains or loses mixing properties
after a time reparameterization.
Note that the mixing properties are in general not preserved~\cite{Fay:02}, since there are analytic time reparameterizations of a completely non-chaotic flow
on ${\mathbb T}^3$ that are mixing (see~\cite{KatTho:p} for additional references).
Other extensions can be envisioned in the joint characterization of flows and their discretizations,
particularly through the consideration of Poincar\'e maps and
suspension flows, for which many of the necessary techniques have been developed.

\section*{Acknowledgments}

The authors thank Aysa Sahin for stimulating discussions and Stefano Galatolo for hints on Lemma~\ref{ddprop}. A.E.M. acknowledges support from the Alfred P. Sloan Foundation in the form of a Sloan Research Fellowship.


\end{document}